\documentclass[11pt]{article}
\usepackage{a4wide}
\usepackage[english]{babel}
\usepackage[utf8]{inputenc}
\usepackage[T1]{fontenc}
\usepackage{amsmath,amssymb,amsthm,stmaryrd}
\usepackage{tikz}
\usepackage{titlefoot}
\usepackage{url}
\usepackage{cleveref}
\usepackage{latexsym}
\usepackage{enumitem}

\RequirePackage{picinpar}

\theoremstyle{plain}
\newtheorem{theorem}{Theorem}[section]

\newtheorem{lemma}[theorem]{Lemma}
\newtheorem{corollary}[theorem]{Corollary}
\newtheorem{proposition}[theorem]{Proposition}

\theoremstyle{definition}
\newtheorem{definition}{Definition}[section]

\newtheoremstyle{dotless}{}{}{\itshape}{}{\bfseries}{}{ }{}
\theoremstyle{dotless}
\newtheorem{thm}{Theorem}

\newcommand{\wcol}{\mathrm{wcol}}
\newcommand{\col}{\mathrm{col}}
\newcommand{\td}{\mathrm{td}}
\newcommand{\tw}{\mathrm{tw}}

\newcommand{\Wreach}{\mathrm{WReach}}
\newcommand{\Sreach}{\mathrm{SReach}}

\newcommand{\Oof}{\mathcal{O}}
\newcommand{\Omof}{\Omega}

\newcommand{\N}{\mathbb{N}}
\newcommand{\ie}{i.e.\ }
\newcommand{\GHi}[1]{G[H_{\ge#1}]}

\title{On the Generalised Colouring Numbers\\[1mm]
  of Graphs that Exclude a Fixed Minor}

\author{Jan van den Heuvel\,\thanks{\,Department of Mathematics, London
    School of Economics and Political Science, London, United Kingdom;
    \texttt{j.van-den-heuvel@lse.ac.uk}, \texttt{d.quiroz@lse.ac.uk}.} \and
  Patrice Ossona de Mendez\,\thanks{\,Centre d'Analyse et de
    Math\'ematiques Sociales (CNRS, UMR 8557), Paris, France, and Computer
    Science Institute of Charles University (IUUK), Prague, Czech Republic.
    Supported by grant ERCCZ LL-1201 and CE-ITI P202/12/G061, and by the
    European Associated Laboratory ``Structures in Combinatorics'' (LEA
    STRUCO); \texttt{pom@ehess.fr}.}
  \and Daniel Quiroz\,\footnotemark[1]\\[1mm]
  \and Roman Rabinovich\,\thanks{\,Logic and Semantics, Technical
    University Berlin, Germany; \texttt{roman.rabinovich@tu-berlin.de}. Supported by the European Research Council (ERC) under the European Union’s Horizon 2020 research and innovation program (ERC consolidator grant DISTRUCT, agreement No. 648527)}
  \and Sebastian Siebertz\,\thanks{\,Institute of Informatics, Faculty of
    Mathematics, Informatics and Mechanics, University of Warsaw, Poland;
    \texttt{siebertz@mimuw.edu.pl}.}}

\begin{document}
\date{}
\maketitle

\unmarkedfntext{\begin{window}[0,r,{\includegraphics[height=6mm]{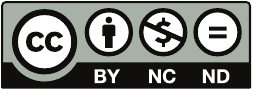}},{}]
    \noindent This work has been submitted to the \emph{European Journal of
      Combinatorics}. It is licensed under the Creative Commons
    Attribution-NonCommercial-NoDerivatives 4.0 International License.\\
    To view a copy of this license, visit
    \url{creativecommons.org/licenses/by-nc-nd/4.0/}.
  \end{window}}

\begin{abstract}
  \noindent
  The generalised colouring numbers $\col_r(G)$ and $\wcol_r(G)$ were
  introduced by Kierstead and Yang as a generalisation of the usual
  colouring number, and have since then found important theoretical and
  algorithmic applications.

  In this paper, we dramatically improve upon the known upper bounds for
  generalised colouring numbers for graphs excluding a fixed minor, from
  the exponential bounds of Grohe \emph{et al.}\ to a linear bound for the
  $r$-colouring number $\col_r$ and a polynomial bound for the weak
  $r$-colouring number $\wcol_r$. In particular, we show that if $G$
  excludes~$K_t$ as a minor, for some fixed $t\ge4$, then
  $\col_r(G)\le\binom{t-1}{2}\,(2r+1)$ and
  $\wcol_r(G)\le \binom{r+t-2}{t-2}\cdot\linebreak(t-3)(2r+1)\in
  \Oof(r^{\,t-1})$.

  In the case of graphs $G$ of bounded genus $g$, we improve the bounds to
  $\col_r(G)\le(2g+3)(2r+1)$ (and even $\col_r(G)\le5r+1$ if $g=0$, \ie if
  $G$ is planar) and $\wcol_r(G)\le\Bigl(2g+\binom{r+2}{2}\Bigr)\,(2r+1)$.


\end{abstract}

\clearpage
\section{Introduction}

The \emph{colouring number $\col(G)$} of a graph $G$ is the minimum integer
$k$ such that there is a strict linear order $<_L$ of the vertices of $G$
for which each vertex $v$ has \emph{back-degree} at most $k-1$, \ie at most
$k-1$ neighbours $u$ with $u<_Lv$. It is well-known that for any graph $G$,
the chromatic number~$\chi(G)$ satisfies $\chi(G)\le \col(G)$.

Some generalisations of the colouring number of a graph have been studied
in the literature. These include the \emph{arrangeability} \cite{chen93}
used in the study of Ramsey numbers of graphs, the \emph{admissibility}
\cite{kierstead94}, and the \emph{rank} \cite{kierstead2000} used in the
study of the game chromatic number of graphs. But maybe the most natural
generalisation of the colouring numbers is the two series $\col_r$ and
$\wcol_r$ of \emph{generalised colouring numbers} introduced by Kierstead
and Yang \cite{kierstead03} in the context of colouring games and marking
games on graphs. As proved by Zhu \cite{zhu2009colouring}, these invariants
are strongly related to low tree-depth decompositions
\cite{nevsetvril2006tree}, and can be used to characterise bounded
expansion classes of graphs (introduced in \cite{POMNI}) and nowhere dense
classes of graphs (introduced in \cite{nevsetvril2011nowhere}). For more
details on this connection, we refer the interested reader
to~\cite{NesetrilOdM12}.

The invariants $\col_r$ and $\wcol_r$ are defined in a way similar to the
usual definition of the colouring number: the \emph{$r$-colouring number
  $\col_r(G)$} of a graph $G$ is the minimum integer $k$ such that there is
a linear order $<_L$ of the vertices for which each vertex $v$ can reach at
most $k-1$ other vertices smaller than~$v$ (in the order $<_L$) with a path
of length at most $r$, all internal vertices of which are greater than $v$.
For the \emph{weak $r$-colouring number $\wcol_r(G)$}, we do not require
that the internal vertices are greater than~$v$, but only that they are
greater than the final vertex of the path. (Formal definitions will be
given in Section~\ref{sec:preliminaries}.) As noticed already
in~\cite{kierstead03}, the two types of generalised colouring numbers are
related by the inequalities
\begin{equation*}
  \col_r(G)\:\le\: \wcol_r(G)\:\le\: (\col_r(G))^r.
\end{equation*}
If we allow paths of any length (but still restrictions on the position of
the internal vertices), we get the \emph{$\infty$-colouring number
  $\col_\infty(G)$} and the \emph{weak $\infty$-colouring number
  $\wcol_\infty(G)$}.

Generalised colouring numbers are an important tool in the context of
algorithmic sparse graphs theory. They play a key role for example in the
model-checking and enumeration algorithms for first-order logic on bounded
expansion and nowhere dense graph
classes~\cite{dvovrak2013testing,grohe2014deciding,kazana2013enumeration},
in Dvo\v{r}\'ak's linear time approximation algorithm for minimum
distance-$r$ dominating sets~\cite{dvovrak2013constant}, and in the
kernelisation algorithms for distance-$r$ dominating
sets~\cite{Drange16,eickmeyer2016neighborhood}.

An interesting aspect of generalised colouring numbers is that these
invariants can also be seen as gradations between the colouring number
$\col(G)$ and two important minor monotone invariants, namely the
\emph{tree-width $\tw(G)$} and the \emph{tree-depth $\td(G)$} (which is the
minimum height of a depth-first search tree for a supergraph of $G$
\cite{nevsetvril2006tree}). More explicitly, for every graph~$G$ we have
the following relations.

\begin{proposition}\label{pro:coltwtd}\mbox{}\\*
  \hbox to1.5\parindent{\rm (a)\hss}$\col(G)= \col_1(G)\le \col_2(G)\le
  \dots\le \col_\infty(G)= \tw(G)+1;$

  \smallskip\noindent
  \hbox to1.5\parindent{\rm (b)\hss}$\col(G)= \wcol_1(G)\le \wcol_2(G)\le
  \dots\le \wcol_\infty(G)= \td(G).$
\end{proposition}

\noindent
The equality $\col_\infty(G)=\tw(G)+1$ was first proved in~\cite{Grohe15};
for completeness we include the proof in Subsection~\ref{sub2.2}. The
equality $\wcol_\infty(G)=\td(G)$ is proved in
\cite[Lemma~6.5]{NesetrilOdM12}.

As tree-width~\cite{halin1976s} is a fundamental graph invariant with many
applications in graph structure theory, most prominently in Robertson and
Seymour's theory of graphs with forbidden minors~\cite{robertsonseymour},
it is no wonder that the study of generalised colouring numbers might be of
special interest in the context of proper minor closed classes of graphs.
As we shall see, excluding a minor indeed allows us to prove strong upper
bounds for the generalised colouring numbers.

Using probabilistic arguments, Zhu~\cite{zhu2009colouring} was the first to
give a non-trivial bound for $\col_r(G)$ in terms of the densities of
shallow minors of $G$. For a graph $G$ excluding a complete graph~$K_t$ as
a minor, Zhu's bound gives
\begin{equation*}
  \col_r(G)\:\le\: 1+q_r,
\end{equation*}
where $q_1$ is the maximum average degree of a minor of $G$, and $q_i$ is
inductively defined by $q_{i+1}=q_1\cdot q_i^{2i^2}$.

Grohe \emph{et al.}\ \cite{Grohe15} improved Zhu's bounds as follows:
\begin{equation*}
  \col_r(G)\:\le\: (crt)^{r},
\end{equation*}
for some (small) constant $c$ depending on $t$.

Our main results is an improvement of those bounds for the generalised
colouring numbers of graphs excluding a minor.

\begin{theorem}\label{thm:col}\mbox{}\\*
  Let $H$ be a graph and $x$ a vertex of $H$. Set $h=|E(H-x)|$, and let
  $\alpha$ be the number of isolated vertices of $H-x$. Then for every
  graph $G$ that excludes $H$ as a minor, we have
  \begin{equation*}
    \col_r(G)\:\le\: h\cdot(2r+1)+\alpha.
  \end{equation*}
\end{theorem}  

\noindent
For classes of graphs that are defined by excluding a complete graph~$K_t$
as a minor, we get the following special result.

\begin{corollary}\label{col:col}\mbox{}\\*
  For every graph $G$ that excludes the complete graph $K_t$ as a minor, we
  have
  \begin{equation*}
    \col_r(G)\:\le\: \binom{t-1}{2}\cdot(2r+1).
  \end{equation*}
\end{corollary}

\noindent
For the weak $r$-colouring numbers we obtain the following bound.

\begin{theorem}\label{thm:wcol}\mbox{}\\*
  Let $t\ge4$. For every graph $G$ that excludes $K_t$ as a minor, we have
  \begin{equation*}
    \wcol_r(G)\:\le\: 
    \binom{r+t-2}{t-2}\cdot (t-3)(2r+1)\in \Oof(r^{\,t-1}).
  \end{equation*}
\end{theorem}

\noindent
We refrain from stating a bound on the weak $r$-colouring numbers in the
case that a general graph~$H$ is excluded as minors for conceptual
simplicity. It will be clear from the proof that if a proper subgraph of
$K_t$ is excluded, the bounds can be slightly improved. Those improvements,
however, will only be linear in $t$.

The \emph{acyclic chromatic number~$\chi_a(G)$} of a graph~$G$ is the
smallest number of colours needed for a proper vertex-colouring of~$G$ such
that every cycle has at least three colours. The best known upper bound for
the acyclic chromatic number of graphs without a $K_t$-minor is
$\Oof(t^2\log^2\!t)$, implicit in~\cite{nodm03}. Kierstead and Yang
\cite{kierstead03} gave a short prove that $\chi_a(G)\le\col_2(G)$.
Corollary~\ref{col:col} shows that for graphs~$G$ without a $K_t$-minor we
have $\col_2(G)\in\Oof(t^2)$, which immediately gives an improved
$\Oof(t^2)$ upper bound for the acyclic chromatic number of those graphs as
well.

In the particular case of graphs with bounded genus, we can improve
our bounds further.

\begin{theorem}\label{thm:colg}\mbox{}\\*
  For every graph $G$ with genus $g$, we have
  $\displaystyle\col_r(G)\le (4g+5)r+2g+1$.

  \smallskip\noindent
  In particular, for every planar graph $G$, we have
  $\displaystyle\col_r(G)\le 5r+1$.
\end{theorem}

\begin{theorem}\label{thm:wcolg}\mbox{}\\*
  For every graph $G$ with genus $g$, we have
  $\displaystyle\wcol_r(G)\le \biggl(2g+\binom{r+2}{2}\biggr)\cdot (2r+1)$.

  \smallskip\noindent
  In particular, for every planar graph $G$, we have
  $\displaystyle\wcol_r(G)\:\le\: \binom{r+2}{2}\cdot (2r+1)$.
\end{theorem}

\noindent
For planar graphs, the bound on $\col_1(G)=\wcol_1(G)=\col(G)$ is best
possible. Also for $t=2,3$ and $r=1$ one can easily give best possible
bounds, as expressed in the following observations.

\begin{proposition}\label{pro:wcol}\mbox{}\\*
  \hbox to1.5\parindent{\rm (a)\hss}For every graph $G$ that excludes $K_2$
  as a minor, we have $\col_r(G)=\wcol_r(G)=1$.

  \smallskip\noindent
  \hbox to1.5\parindent{\rm (b)\hss}For every graph $G$ that excludes $K_3$
  as a minor, we have $col_r(G)\le2$ and $\wcol_r(G)\le{}$\\
  \mbox{}\hbox to1.5\parindent{\hss}$r+1$.

  \smallskip\noindent
  \hbox to1.5\parindent{\rm (c)\hss}For every graph $G$ that excludes $K_t$
  as a minor, $t\ge4$, we have
  \begin{equation*}
    \col_1(G)=\wcol_1(G)\:\le\: (0.64+o(1))\,t\sqrt{\ln t}+1\quad
    (|V(G)|\to\infty).
  \end{equation*}
\end{proposition}

\noindent
Part~(a) in the proposition is a triviality. For part (b), note that
excluding $K_3$ as a minor means that~$G$ is acyclic, hence a forest, and
that in this case it is obvious that $\col_r(G)\le2$ and
$\wcol_r(G)\le r+1$. Finally, $\col_1(G)=\wcol_1(G)$ is one more than the
degeneracy of $G$, thus part~(c) follows from Thomason's bound for the
average degree of graphs with no $K_t$ as a minor \cite{Th2}.

Regarding the sharpness on our upper bounds in the results above, we can
make the following remarks.

\mbox{$\bullet$ \ \ }Lower bounds for the generalised colouring numbers for
minor closed classes are given in~\cite{Grohe15}. In that paper it is shown
that for every $k$ and every $r$ there is a graph $G_{k,r}$ of tree-width
$k$ that satisfies $\col_r(G_{k,r})=k+1$ and
$\wcol_r(G_{k,r})=\binom{r+k}{k}$. Graphs of tree-width~$k$
exclude~$K_{k+2}$ as a minor. This shows that our results for classes with
excluded minors are optimal up to a factor $(t-1)\,(2r+1)$.

\mbox{$\bullet$ \ \ }Since graphs with tree-width 2 are planar, this also
shows that there exist planar graphs~$G$ with
$\wcol_r(G)=\binom{r+2}{2}\in \Omof(r^2)$. Compare this to the upper bound
$\wcol_r(G)\in\Oof(r^3)$ for planar graphs in Theorem~\ref{thm:wcolg}.

\mbox{$\bullet$ \ \ }It follows from Proposition~\ref{pro:coltwtd}\,(a)
that a minor closed class of graphs has uniformly bounded colouring number
if and only if it has bounded tree-width. For classes with unbounded
tree-width, such a uniform bound cannot be expected. By analysing the shape
of admissible paths, it is possible to prove that the planar $r\times r$
grid $G_{r\times r}$ satisfies $\col_r(G_{r\times r})\in \Omof{(r)}$. This
shows that for planar graphs~$G$, a best possible bound for $\col_r(G)$
will be linear in~$r$.

\mbox{$\bullet$ \ \ }It follows from \cite[Lemma 3.3]{zhu2009colouring}
that for $3$-regular graphs of high girth the weak \mbox{$r$-colouring}
numbers grow exponentially with $r$. Hence the polynomial bound for
$\wcol_r(G)$ in \Cref{thm:col} for classes with excluded minors cannot be
extended to classes with bounded degree, or even to classes with excluded
topological minors.

\smallskip
The structure of this paper is as follows. In the next section we give
necessary definitions, and prove the connections between the generalised
colouring numbers and tree-width. In Section~\ref{sec:ipd} we introduce
\emph{flat decompositions}, which is our main tool in proving our results,
and give an upper bound for the minimum width of a flat decomposition of a
graph excluding a complete minor. In Section~\ref{sec:wcolmain} we prove
Theorem~\ref{thm:wcol} and in Section~\ref{sec:main} we prove
Theorem~\ref{thm:col}. Our proofs will rely on the notion of the
\emph{elimination-width} of a vertex-order $<_L$, and its connection to
weak colouring, stated as Theorem~\ref{thm:tw}, which was proved in
\cite{Grohe15}. In Section~\ref{sec:wcolplanar} we prove
Theorems~\ref{thm:colg} and~\ref{thm:wcolg}, which have a detailed analysis
of the generalised colouring numbers of planar graphs at their base.

\section{Preliminaries}
\label{sec:preliminaries}

All graphs in this paper are finite, undirected and simple, that is, they
do not have loops or multiple edges between the same pair of vertices. For
a graph $G$, we denote by $V(G)$ the vertex set of $G$ and by $E(G)$ its
edge set.

The \emph{distance} between a vertex $v$ and a vertex $w$ is the length
(that is, the number of edges) of a shortest path between $v$ and $w$. For
a vertex $v$ of $G$, we write $N^G(v)$ for the set of all neighbours of
$v$, $N^G(v)=\{\,u\in V(G)\mid \{u,v\}\in E(G)\,\}$, and for $r\in\N$ we
denote by $N_r^G[v]$ the \emph{closed $r$-neighbourhood of $v$}, that is,
the set of vertices of $G$ at distance at most $r$ from $v$. Note that we
always have $v\in N^G_r[v]$. When no confusion can arise regarding the
graph $G$ we are considering, we usually omit the superscript~$G$.

Let $M$ be a graph with vertices $h_1,\ldots, h_n$. The graph $M$ is a
\emph{minor} of a graph $G$ if in~$G$ there are disjoint connected
subgraphs $H_1,\ldots,H_n$ such that if $\{h_i, h_{j}\}$ is an edge of $M$,
then~$H_i$ is connected to $H_j$ (in $G$). We call the subgraphs
$H_1,\ldots,H_n$ of $G$ a \emph{model} of $M$ in~$G$.

\subsection{Generalised Colouring Numbers}

Let $\Pi(G)$ be the set of all linear orders of the vertices of the graph
$G$, and let $L\in\Pi(G)$. For readability, we write $u<_L v$ if $u$ is
smaller than $v$ with respect to $L$, and $u\le_L v$ if $u<_L v$ or $u=v$.

Let $u,v\in V(G)$. For a positive integer $r$, we say that $u$ is
\emph{weakly $r$-reachable} from~$v$ with respect to~$L$, if there exists a
path $P$ of length $\ell$, $0\le\ell\le r$, between $u$ and $v$ such that
$u$ is minimum among the vertices of $P$ (with respect to $L$). Let
$\Wreach_r[G,L,v]$ be the set of vertices that are weakly $r$-reachable
from~$v$ with respect to $L$. Note that $v\in\Wreach_r[G,L,v]$.

If we allow paths of any length, then we call $u$ \emph{weakly reachable}
from~$v$ with respect to~$L$, and the set of such vertices is denoted by
$\Wreach_\infty[G,L,v]$

Next, $u$ is \emph{strongly $r$-reachable} from $v$ with respect to~$L$, if
there is a path $P$ of length $\ell$, $0\le\ell\le r$, connecting $u$ and
$v$ such that $u\le_Lv$ and such that all inner vertices $w$ of $P$ satisfy
$v<_Lw$. Let $\Sreach_r[G,L,v]$ be the set of vertices that are strongly
$r$-reachable from~$v$ with respect to $L$. Note that again we have
$v\in \Sreach_r[G,L,v]$.

Again, if we allow paths of any length, then we say that $u$ is
\emph{strongly reachable} from $v$, and the collection of all such vertices
is denoted $\Sreach_\infty[G,L,v]$.

For $r\in\mathbb{N}\cup\{\infty\}$, the \emph{weak $r$-colouring number
  $\wcol_r(G)$} of $G$ is defined as
\begin{equation*}
  \wcol_r(G):=\min_{L\in\Pi(G)}\:\max_{v\in V(G)}\:
  \bigl|\Wreach_r[G,L,v]\bigr|,
\end{equation*}
and the \emph{$r$-colouring number $\col_r(G)$} of $G$ is defined as
\begin{equation*}
  \col_r(G):=\min_{L\in\Pi(G)}\:\max_{v\in V(G)}\:
  \bigl|\Sreach_r[G,L,v]\bigr|.
\end{equation*}

\subsection{Tree-width and elimination width}\label{sub2.2}

The concept of tree-width has shown itself to be very useful for the design
of efficient graph algorithms. Many NP-hard problems are fixed-parameter
tractable when parametrised by the tree-width of the input graph. A very
general theorem due to Courcelle~\cite{cou90} states that every problem
definable in monadic second-order logic can be solved in linear time on a
class of graphs of bounded tree-width.

The most common definition of tree-width is in terms of
tree-decompositions. A \emph{tree-decomposition} of a graph $G$ is a pair
$\left(T,(X_t)_{t\in V(T)}\right)$, where $T$ is a tree and
$X_t\subseteq V(G)$ for each $t\in V(T)$, such that
\begin{enumerate}[itemsep=0pt, topsep=1pt]
\item $\bigcup_{t\in V(T)}=V(G)$;
\item for every edge $\{u,v\}\in E(G)$, there is a $t\in V(T)$ such that
  $u,v\in X_t$; and
\item if $v\in X_t\cap X_{t'}$ for some $t,t'\in V(T)$, then $v\in X_{t''}$
  for all $t''$ that lie on the unique path between $t$ and $t'$ in $T$.
\end{enumerate}
The width of a tree-decomposition is $\max_{t\in V(T)}|X_t|-1$, and the
\emph{tree-width of $G$} is equal to the smallest width of any
tree-decomposition of $G$.

For a linear order $L\in\Pi(G)$, the \emph{fill-in of $G$ with respect to
  $L$} is the graph $G_L$ obtained by inductively adding for each vertex
$v$ (starting with the largest vertex of the order) an edge $\{u,w\}$ for
all $u,w\in N(v)$, $u\ne w$, with $u<_Lv$ and $w<_Lv$. An equivalent
definition of~$G_L$ would be the graph obtained by making each vertex $v$
adjacent to all the vertices smaller than~$v$ (with respect to~$L$) than
can be reached from~$v$ in~$G$ by a path whose internal vertices are
greater than $v$. The \emph{elimination-width} of an order $L$ is the size
of the largest clique in $G_L$ minus $1$ (\ie equal to $\omega(G_L)-1$,
where~$\omega(G)$ is the clique number of a graph $G$).

It is not so hard to prove (see, e.g., \cite[Theorem~3.1]{arn85}) that the
tree-width of $G$ is equal to the minimum elimination-width over all orders
of $V(G)$:
\begin{equation*}
  \tw(G)=\min_{L\in\Pi(G)}\:\omega (G_L)-1.
\end{equation*}
On the other hand, $\omega(G_L)-1$ obviously is equal to the maximum over
all vertices $v$ in $G$ of the number of vertices smaller than~$v$ that can
be reached from $v$ by a path whose internal vertices are greater than $v$.
(The largest clique in~$G_L$ also includes $v$ itself, which is counted for
$\col_{\infty}(G)$, but not for $\tw(G)$.) This shows that
$\col_\infty(G)=\tw(G)+1$, as was claimed earlier.

We also have that elimination-width is related to weak reachability, as the
next result shows.

\begin{theorem}[Grohe \emph{et al.}\
  \cite{Grohe15}]\label{thm:tw}\mbox{}\\*
  Let $G$ be a graph and let $L\in\Pi(G)$ be a linear order of $V(G)$ with
  elimination-width at most~$k$. For all $r\in\N$ and all $v\in V(G)$, we
  have
  \begin{equation*}
    \bigl|\Wreach_r[G,L,v]\bigr|\:\le\: \binom{r+k}{k}.
  \end{equation*}
\end{theorem}

\section{Flat decompositions}\label{sec:ipd}

Our main tool in proving our results will be flat decompositions, which we
introduce now.

Let $G$ be a graph, let $H\subseteq G$ be a subgraph of $G$, and let
$f\colon\N\to\N$ be a function. We say that~$H$ \emph{$f$-spreads on $G$}
if, for every $r\in\N$ and $v\in V(G)$, we have
\begin{equation*}
  |N^G_r[v]\cap V(H)|\:\le\:\le f(r).
\end{equation*}

Let $H,H'$ be vertex-disjoint subgraphs of $G$. We say that $H$ is
\emph{connected} to $H'$ if some vertex in $H$ has a neighbour in $H'$, \ie
if there is an edge $\{u,v\}\in E(G)$ such $u\in V(H)$ and $v\in V(H')$.

\begin{definition}\mbox{}\\*
  A \emph{decomposition} of a graph~$G$ is a sequence
  $\mathcal{H}=(H_1,\ldots,H_\ell)$ of non-empty subgraphs of $G$ such that
  the vertex sets $V(H_1),\ldots,V(H_\ell)$ partition $V(G)$. The
  decomposition $\mathcal H$ is \emph{connected} if each~$H_i$ is
  connected.
\end{definition}

\noindent
For a decomposition $(H_1,\ldots,H_\ell)$ of a graph~$G$ and
$1\le i\le\ell$, we denote by $\GHi{i}$ the subgraph of~$G$ induced by
$\bigcup_{i\le j\le\ell}V(H_j)$.

\begin{definition}\mbox{}\\*
  We call the decomposition $\mathcal H$ \emph{$f$-flat} if each $H_i$
  $f$-spreads on $\GHi{i}$.

  \smallskip\noindent
  A \emph{flat decomposition} is a decomposition that is $f$-flat
    for some function $f\colon\N\to\N$.
\end{definition}

\begin{definition}\mbox{}\\*
  Let $\mathcal{H}=(H_1,\ldots,H_\ell)$ be a decomposition of a graph $G$,
  let $1\le i<\ell$, and let $C$ be a component of~$\GHi{(i+1)}$. The
  \emph{separating number of $C$} is the maximal number~$s$ of (distinct)
  graphs $Q_1,\ldots,Q_s\in\{H_1,\ldots,H_i\}$ such that all the $Q_j$'s
  are connected to~$C$.
\end{definition}

\noindent
Note that the separating number of a component~$C$ is independent of the
value~$i$ such that~$C$ is a component of $\GHi{(i+1)}$. Indeed, let $i$ be
minimal such that $C$ is a component of $\GHi{(i+1)}$. Then for all $t>i$
we have that either $H_t$ is not connected to~$C$, or $H_t$ is a subgraph
that contains vertices from $C$.

\begin{definition}\mbox{}\\*
  Let $\mathcal{H}=(H_1,\ldots,H_\ell)$ be a decomposition of a graph $G$.
  The \emph{width} of $\mathcal{H}$ is the maximum separating number of a
  component~$C$ of $\GHi{i}$, maximised over all $i$, $1\le i<\ell$.
\end{definition}

\noindent
We call a path $P$ in $G$ an \emph{isometric path} if $P$ is a shortest
path between its endpoints. Isometric paths will play an important role in
the analysis of flat decompositions and the generalised colouring numbers.
We call a flat decomposition $\mathcal{H}=(H_1,\ldots, H_\ell)$ an
\emph{isometric paths decomposition} if each $H_i$ is an isometric path in
$\GHi{i}$.

A definition similar to isometric paths decompositions is given in
\cite{abraham14}, where they are called \emph{cop-decompositions}. The name
\emph{cop-decomposition} in \cite{abraham14} is inspired by a result of
\cite{andreae86}, which shows that such decompositions of small width exist
for classes of graphs that exclude a fixed minor, and which uses a
cops-and-robber game argument. The difference between a cop-decomposition
and a connected decomposition is that in a connected decomposition we allow
arbitrary connected subgraphs rather than just paths as in a
cop-decomposition

The property of having a partition into connected subgraphs with the above
width properties is extremely useful, as it allows us to contract the
subgraphs to find a minor of $G$ with bounded tree-width, as expressed in
the following lemma.

\begin{lemma}\label{lem:3.1}\mbox{}\\*
  Let $G$ be a graph, and let $\mathcal{H}=(H_1,\ldots,H_\ell)$ be a
  connected decomposition of $G$ of width $k$. By contracting each
  connected subgraph $H_i$ to a single vertex, we obtain a graph
  $H=G/\mathcal{H}$ with~$\ell$ vertices and tree-width at most $k$.
\end{lemma}

\begin{proof}
  We identify the vertices of~$H$ with the connected subgraphs
  $\{H_1,\ldots,H_\ell\}$. By the contracting operation, two subgraphs
  $H_i,H_j$ are adjacent in $H$ if there is an edge in $G$ between a vertex
  of $H_i$ and a vertex of~$H_j$, and there is a path
  $H_i,H_{i+1},\ldots,H_j$ in $H$ if and only if there is a path between
  some vertex of $H_i$ and some vertex of~$H_j$ that uses only vertices of
  $H_i, H_{i+1}, \ldots, H_j$, in that order.

  Let~$L$ be the order of $V(H)$ given by the order of the subgraphs in the
  connected decomposition. Consider the graph $H_L$, the fill-in of $H$
  with respect to $L$. For any vertex $H_i$ of~$H$, the set of neighbours
  of $H_i$ in $H_L$ that are smaller than $H_i$ (with respect to $L$) is
  the set of subgraphs among $H_1,\ldots, H_{i-1}$ that are reachable via a
  path (in $H$) with internal vertices larger than $H_i$. As each such path
  corresponds to a path in $G$ as described above, this is exactly the set
  of subgraphs in $\{H_1,\ldots,H_{i-1}\}$ that are reachable in $G$ from
  the component $C$ of $\GHi{i}$ that contains $H_i$. The number of such
  subgraphs is the separating number of~$C$, which by definition of the
  width of $\mathcal{H}$ is at most $k$. Since $H_i$ is also strongly
  reachable from itself, we see that
  $\bigl|\Sreach_\infty[H,L,H_i]\bigr|\le k+1$ for all $H_i\in V(H)$. This
  shows that $\tw(H)+1=\col_\infty(H)\le k+1$, as required.
\end{proof}

\noindent
A fundamental property of isometric paths is that from any vertex
$v$, not many vertices of an isometric path can be reached from $v$ in $r$
steps.

\begin{lemma}\label{lem:shortestpath}\mbox{}\\*
  Let $v$ be a vertex of a graph $G$, and let $P$ be an isometric path in
  $G$. Then $P$ contains at most $2r+1$ vertices of the closed
  $r$-neighbourhood of $v$:
  $|N_r[v]\cap V(P)|\le \min\{\,|V(P)|,\,2r+1\,\}$.
\end{lemma}

\begin{proof}
  Assume $P=v_0,\ldots,v_n$ and $\bigl|N_r[v]\cap V(P)\bigr|>2r+1$. Let $i$
  be minimal such that $v_i\in N_r[v]$ and let $j$ be maximal such that
  $v_j\in N_r[v]$. As $P$ is a shortest path, the distance in $G$ between
  $v_i$ and $v_j$ is $j-i\ge\bigl|N_r[v]\cap V(P)\bigr|-1>2r$, which
  contradicts the hypothesis that both $v_i$ and $v_j$ are at distance at
  most $r$ from $v$, thus at distance at most $2r$ from each other.
\end{proof}

\noindent
From a decomposition $(H_1,\ldots,H_\ell)$ of a graph~$G$, we define a
linear order $L$ on $V(G)$ as follows. First choose an arbitrary linear
order on the vertices of each subgraph~$H_i$. Now let~$L$ be the linear
extension of that order where for $v\in V(H_i)$ and $w\in V(H_j)$ with
$i<j$ we define $L(v)<L(w)$.

\begin{lemma}\label{lem:deletevertices}\mbox{}\\*
  Let $\mathcal{H}=(H_1,\ldots, H_\ell)$ be a decomposition of a graph $G$,
  and let $L$ be an order defined from the decomposition. For an integer
  $i$, $1\le i\le\ell$, let $G'= \GHi{i}$. Then we have for every $r\in\N$
  and every $v\in V(G)$:
  \begin{eqnarray*}
    \Sreach_r[G,L,v]\cap V(H_i)\:\subseteq\: N_r^{G'}[v]\cap V(H_i),\\
    \Wreach_r[G,L,v]\cap V(H_i)\:\subseteq\: N_r^{G'}[v]\cap V(H_i).
  \end{eqnarray*}
\end{lemma}

\begin{proof}
  If a path $P$ with one endpoint $v$ visits a vertex that is smaller than
  a vertex of $H_i$, then the path cannot be continued to weakly or
  strongly visit a vertex of $H_i$.
\end{proof}

\noindent
Now we are in a position to give upper bounds of $\col_r(G)$ and
$\wcol_r(G)$ in terms of the width of a flat decomposition.

\begin{lemma}\label{lem:spd}\mbox{}\\*
  Let $f\colon\N\to \N$ and let $r,k\in\N$. Let $G$ be a graph that admits
  an $f$-flat decomposition of width~$k$. Then we have
  \begin{equation*}
    \col_r(G)\:\le\: (k+1)\cdot f(r).
  \end{equation*}
\end{lemma}

\begin{proof}
  Let $\mathcal{H}=(H_1,\ldots, H_\ell)$ be an $f$-flat decomposition of
  $G$ of width $k$, and let $L$ be a linear order defined from the
  decomposition. Let $v\in V(G)$ be an arbitrary vertex and choose~$q$ such
  that $v\in V(H_{q+1})$. Let $C$ be the component of $\GHi{(q+1)}$ that
  contains $v$, and let $Q_1,\ldots,Q_m$, $1\le m\le q$, be the subgraphs
  among $H_1,\ldots,H_q$ that have a connection to~$C$. Since $\mathcal{H}$
  has width~$k$, we have $m\le k$. By definition of~$L$, the vertices in
  $\Sreach_r[G,L,v]$ can only lie on $Q_1,\ldots,Q_m$ and on $H_{q+1}$,
  hence on at most $k+1$ subgraphs. For $j=1,\ldots,m$, assume that
  $Q_j=H_{i_j}$ and let $G_j'=\GHi{i_j}$. Then by \Cref{lem:deletevertices}
  we have $\Sreach_r[G,L,v]\cap Q_j\subseteq N_r^{G'_j}[v]\cap Q_j$. Since
  $H_{i_j}=Q_j$ $f$-spreads on $G_j'$, we have
  $\bigl|N_r^{G'_j}[v]\cap Q_j\bigr|\le f(r)$. The result follows.
\end{proof}

\begin{lemma}\label{lem:spdwcol}\mbox{}\\*
  Let $f\colon\N\to\N$ and let $r,k\in\N$. Let $G$ be a graph that admits a
  connected $f$-flat decomposition of width $k$. Then we have
  \begin{equation*}
    \wcol_r(G)\:\le\: \binom{r+k}{k}\cdot f(r).
  \end{equation*}
\end{lemma}

\begin{proof}
  Let $\mathcal{H}=(H_1,\ldots, H_\ell)$ be a connected $f$-flat
  decomposition of width $k$, and let~$L$ be a linear order defined from
  it. We contract the subgraphs $H_1,\ldots, H_\ell$ to obtain a graph $H$
  of tree-width at most~$k$ (see Lemma~\ref{lem:3.1}). We identify the
  vertices of $H$ with the subgraphs~$H_i$. For a vertex $v\in V(G)$,
  consider the subgraph $H_i$ with $v\in V(H_i)$. By \Cref{thm:tw}, the
  vertex~$H_i$ weakly $r$-reaches at most $\binom{r+k}{k}$ vertices in $H$
  that are smaller than or equal to $H_i$ in the order on $V(H)$ induced by
  $L$. These vertices $H_j$ that are weakly $r$-reachable from~$H_i$ in~$H$
  are the only subgraphs in $G$ that may contain vertices that are weakly
  $r$-reachable from~$v$ in~$G$. We conclude that there are at most
  $\binom{r+k}{k}$ subgraphs among $H_1,\ldots,H_\ell$ in~$G$ that contain
  vertices that are weakly $r$-reachable from~$v$. As in the previous proof
  we can argue that there are at most $f(r)$ weakly $r$-reachable vertices
  on each subgraph, which completes the proof.
\end{proof}

\section{The weak $r$-colouring numbers of graphs excluding a fixed
  complete minor}\label{sec:wcolmain}

In this section we prove Theorem~\ref{thm:wcol}. We will provide a more
detailed analysis for the \mbox{$r$-colouring} numbers in the next section.

\begingroup
\def\thethm{\relax}
\begin{thm} \ {\rm(Theorem~\ref{thm:wcol})}\mbox{}\\*
  Let $t\ge4$. For every graph $G$ that excludes $K_t$ as a minor, we have
  \begin{equation*}
    \wcol_r(G)\:\le\: 
    \binom{r+t-2}{t-2}\cdot (t-3)(2r+1)\in \Oof(r^{\,t-1}).
  \end{equation*}
\end{thm}
\endgroup

\noindent
Theorem \ref{thm:wcol} is a direct consequence of \Cref{lem:spdwcol} and of
\Cref{lem:minordecompnew}. This lemma states that connected flat
decompositions of small width exist for graphs that exclude a fixed
complete graph $K_t$ as a minor. This result is inspired by the result on
cop-decompositions presented in~\cite{andreae86}.

\begin{lemma}\label{lem:minordecompnew}\mbox{}\\*
  Let $t\ge4$ and let $f\colon\N\to \N$ be the function $f(r)=(t-3)(2r+1)$.
  Let $G$ be a graph that excludes~$K_t$ as a minor. Then there exists a
  connected $f$-flat decomposition of $G$ of width at most $t-2$.
\end{lemma}

\begin{proof}
  Without loss of generality we may assume that $G$ is connected. We will
  iteratively construct a connected $f$-flat decomposition
  $H_1,\ldots,H_\ell$ of $G$. For all $q$, $1\le q<\ell$, we will maintain
  the following invariant. Let $C$ be a component of $\GHi{(q+1)}$. Then
  the subgraphs $Q_1,\ldots,Q_s\in \{H_1,\ldots,H_q\}$ that are connected
  to $C$ form a minor model of the complete graph $K_s$, for some
  $s\le t-2$. This will immediately imply our claim on the width of the
  decomposition.

  To start, we choose an arbitrary vertex $v\in V(G)$ and let $H_1$ be the
  connected subgraph~$G[v]$. Clearly, $H_1$ $f$-spreads on $G$, and the
  above invariant holds (with $s=1$).

  Now assume that for some $q$, $1\le q\le\ell-1$, the sequence
  $H_1,\ldots,H_q$ has already been constructed. Fix some component $C$ of
  $\GHi{(q+1)}$ and assume that the subgraphs
  $Q_1,\ldots,Q_s\in \{H_1,\ldots, H_q\}$ that have a connection to $C$
  form a minor model of $K_s$, for some $s\le t-2$. Because~$G$ is
  connected, we have $s\ge1$. Let $v$ be a vertex of $C$ that is adjacent
  to a vertex of $Q_1$. Let $T$ be a breadth-first search tree in $G[C]$
  with root $v$. We choose $H_{q+1}$ to be a minimal connected subgraph of
  $T$ that contains $v$ and that contains for each $i$, $1\le i\le s$, at
  least one neighbour of $Q_i$.
  
  It is easy to see that for every component $C'$ of $\GHi{(q+2)}$, the
  subgraphs $Q_1,\ldots,Q_{s'}\in \{H_1,\ldots,H_{q+1}\}$ that are
  connected to $C'$ form a minor model of a complete graph $K_{s'}$, for
  some $s'\le t-1$. Let us show that in fact we have $s'\le t-2$. Towards a
  contradiction, assume that there are $Q_1,\ldots,Q_{t-1}\in
  \{H_1,\ldots,H_{q+1}\}$ that have a connection to $C'$ and such that
  the~$Q_i$ form a minor model of $K_{t-1}$. As each $Q_i$ has a connection
  to $C'$, we can contract the whole component $C'$ to find $K_t$ as a
  minor, a contradiction.

  Let us finally show that the decomposition is $f$-flat. We show that the
  newly added subgraph~$H_{q+1}$ $f$-spreads on $\GHi{(q+1)}$. By
  construction, $H_{q+1}$ is a subtree of $T$ that consists of at most
  $t-3$ isometric paths in $\GHi{(q+1)}$ (possibly not disjoint), since $T$
  is a breadth-first search tree and $v$ is already a neighbour of $Q_1$.
  Now the claim follows immediately from Lemma~\ref{lem:shortestpath}.
\end{proof}

\section{The $r$-colouring numbers of graphs excluding a fixed
  minor}\label{sec:main}
  
For graphs that exclude a complete graph as a minor, we already get a good
bound on the strong $r$-colouring numbers. However, if a sparse graph is
excluded, we can do much better. In this case we will construct an
isometric paths decomposition, where only few paths are separating (in
general, each connected subgraph in our proof may subsume many isometric
paths).

The proof idea is essentially the same as that for
\Cref{lem:minordecompnew}. We will iteratively construct an isometric paths
decomposition $(P_1,\ldots, P_\ell)$ of $G$ such that the components $C$ of
$G[P_{\ge(q+1)}]$ are separated by a minor model of a proper subgraph $M$
of $H-x$. To optimise the bounds on the width of the decomposition, we will
first try to maximise the number of edges in the subgraph $M$, before we
add more vertices to the model. During the construction we will have to
\emph{re-interpret} the separating minor model, as otherwise connections of
a vertex model (the subgraph representing a vertex of~$M$) to the component
may be lost.

To implement the above mentioned re-interpretation of the minor model it
will be more convenient to work with a slightly different (and
non-standard) definition of a minor model. Let $M$ be a graph with vertices
$h_1,\ldots, h_n$. The graph $M$ is a minor of $G$ if there are pairwise in
$G$ disjoint connected subgraphs $H_1,\ldots, H_n$ and pairwise internally
disjoint paths $E_{ij}$ for $\{h_i,h_j\}\in E(M)$ that are also internally
disjoint from the $H_1,\ldots, H_n$, such that if $e_{ij}=\{h_i,h_{j}\}$ is
an edge of $M$, then $E_{ij}$ connects a vertex of $H_i$ with a vertex of
$H_j$. We call the subgraph $H_i$ of $G$ the \emph{model} of $h_i$ in $G$
and the path $E_{ij}$ the \emph{model} of $e_{ij}$ in $G$.

One can easily see that a graph $H$ is a minor of a graph $G$ according to
the definition in Section~\ref{sec:preliminaries} if and only if $H$ is a
minor of $G$ according to the definition given above. The reason to
introduce paths $E_{ij}$ (rather than edges $e_{ij}$) is that we want to
control the number of vertices in vertex models connected to a component.
This is impossible for the connecting paths $E_{ij}$, so it would be
impossible if we let the vertex models grow to encompass the $E_{ij}$.

\begin{lemma}[following \cite{andreae86}]\label{lem:minordecomp}\mbox{}\\*
  Let $H$ be a graph and $x$ a vertex of $H$. Set $h=|E(H-x)|$, and let
  $\alpha$ be the number of isolated vertices of $H-x$. Then every graph
  $G$ that excludes $H$ as a minor admits an isometric paths decomposition
  of width at most $3h+\alpha.$
\end{lemma}

\begin{proof}
  Without loss of generality we may assume that $G$ is connected. Assume
  $H-x$ has vertices $h_1,\ldots,h_k$, $k=|V(H)|-1$. For $1\le i\le k$,
  denote by $d_i$ the degree of $h_i$ in $H-x$.
  
  We will iteratively construct an isometric paths decomposition
  $(P_1,\ldots, P_\ell)$ of $G$. For all~$q$, $1\le q<\ell$, we will
  maintain the four invariants given below. With each component $C$ of
  $G[P_{\ge(q+1)}]$ we associate a minor model of a proper subgraph $M$ of
  $H-x$.

  \begin{enumerate}[itemsep=0pt, topsep=1pt]
  \item\label{it:1} For $h_i\in V(M)$, the models $H_i$ of $h_i$ in $G$ use
    vertices of $P_1,\ldots, P_q$ only.
  \item\label{it:peb} For each $H_i$ with $h_i\in V(M)$ such that $h_i$ is
    an isolated vertex in $H-x$, $H_i$ will consist of a single vertex
    only.

    For each $H_i$ with $h_i\in V(M)$ such that $h_i$ is not an isolated
    vertex in $H-x$, it is possible to place a set of $d_i$ pebbles
    $\{\,p_{ij}\mid\{h_i,h_j\}\in E(H-x)\,\}$ on the vertices of~$H_i$
    (with possibly several pebbles on a vertex), in such a way that the
    pebbles occupy exactly the set of vertices of~$H_i$ with a neighbour in
    $C$. In particular, each $H_i$ has between $1$ and $d_i$ vertices with
    a neighbour in $C$.
  \item\label{it:3} For each edge $e_{ij}=\{h_i, h_j\}\in E(M)$, the model
    $E_{ij}$ of $e_{ij}$ in $G$ has the following properties.
    \begin{enumerate}[itemsep=0pt, topsep=1pt]
    \item The endpoints of $E_{ij}$ are the vertices with pebbles $p_{ij}$
      in $H_i$ and $p_{ji}$ in $H_j$.
    \item\label{it:3singlepath} The internal vertices of $E_{ij}$ belong to
      a single path $P_p$, where $p\le q$.
    \item Assume $E_{ij}$ has internal vertices in $P_p$. Let $D$ be the
      component of $G[P_{\ge p}]$ that contains $P_p$. Let $v_{ij}$ and
      $v_{ji}$ be the vertices of $H_i$ and $H_j$, respectively that are
      pebbled with $p_{ij}$ and $p_{ji}$ (at the time $P_p$ was defined).
      Then $E_{ij}$ is an isometric path in
      $G[D\cup \{v_{ij},v_{ji}\}]-e_{ij}$. (This condition is not necessary
      for the proof of the lemma; it will be used in the proof of
      Theorem~\ref{thm:col}, though.)
    \end{enumerate}
  \item\label{it:all} All vertices on a path of $P_1,\ldots, P_q$ that have
    a connection to $C$ are part of the minor model.
  \end{enumerate}

  \noindent
  Let us first see that maintaining these invariants implies that the
  isometric paths decomposition has the desired width. By
  Condition~\ref{it:all}, the separating number of the component $C$ is
  determined by the number of isometric paths that are part of the minor
  model of $M$ and have a connection to~$C$. To count this number of paths,
  we count the number $m_1$ of paths that lie in any vertex model $H_i$ for
  $h_i\in V(M)$ and have a connection to $C$, and we count the number~$m_2$
  of paths that correspond to the edges $e_{ij}$ of~$M$. By
  Condition~\ref{it:peb}, $m_1$ is at most the number of pebbles in $H$
  plus the number of isolated vertices of $H-x$ . Since the number of
  pebbles of each model $H_i$ is at most $d_i$, the number of pebbles is at
  most the sum of the vertex degrees, and therefore
  $m_1\le 2|E(H-x)|+\alpha$. By Condition~\ref{it:3}(b), $m_2$ is at
  most $|E(H-x)|$. Finally, since $M$ is a proper subgraph of $H-x$, either
  $m_1< 2|E(H-x)|+\alpha$ or $m_2<|E(H-x)|$ and hence we have
  $m_1+m_2< 3|E(H-x)|+\alpha$.

  We show how to construct an isometric paths decomposition with the
  desired properties. To start, we choose an arbitrary vertex $v\in V(G)$
  and let $P_1$ be the path of length $0$ consisting of~$v$ only. For every
  connected component of $G-V(P_1)$, we define $M$ as the single vertex
  graph~$K_1$ and the model $H_1$ of this vertex as $P_1$. All pebbles are
  placed on $v$. As $G$ is connected, we see that Condition~\ref{it:all} is
  satisfied; all other invariants are clearly satisfied.

  Now assume that for some $q$, $1\le q\le\ell-1$, the sequence
  $P_1,\ldots,P_q$ has already been constructed. Fix some component $C$ of
  $G[P_{\ge(q+1)}]$ and assume that the pebbled minor model of a proper
  subgraph $M\subseteq H-x$ with the above properties for $C$ is given. We
  first find an isometric path $P_{q+1}$ that lies completely inside $C$
  and add it to the isometric paths decomposition. The exact choice of
  $P_{q+1}$ depends on which of the following two cases we are in.

  \smallskip
  Case 1: There is a pair $h_i,h_j$ of non-adjacent vertices in $M$ such
  that $\{h_i,h_j\}\in E(H-x)$.\\
  By Condition~\ref{it:peb}, the pebbles~$p_{ij}$ and $p_{ji}$ lie on some
  vertices $v_{ij}$ of $H_i$ and $v_{ji}$ of $H_j$, respectively, that have
  a neighbour in $C$. Let $v_i$ and~$v_j$ be vertices of $C$ with
  $\{v_{ij},v_i\}, \{v_{ji},v_j\}\in E(G)$ (possibly $v_i=v_j$) such that
  the distance between $v_i$ and $v_j$ in $C$ is minimum among all possible
  neighbours of $v_{ij}$ and $v_{ji}$ in $C$. We choose $P_{q+1}$ as an
  arbitrary shortest path in $C$ with endpoints $v_i$ and $v_j$. We add the
  edge $\{h_i,h_j\}$ to $M$ and the path
  $E_{ij}=\{v_{ij},v_i\}+P_{q+1}+\{v_j,v_{ji}\}$ to the model of $M$.

  \smallskip
  Case 2: $M$ is an induced subgraph of $H-x$.\\
  We choose an arbitrary vertex $v\in V(C)$ and define $P_{q+1}$ as the
  path of length $0$ consisting of~$v$ only. We add an isolated vertex
  $h_a$ to $M$, for some $a$ with $1\le a\le k$, such that $h_a$ was not
  already a vertex of $M$ and define $H_a=P_{q+1}$, with any pebbles on
  $v$.

  \smallskip
  Because in both cases the new path $P_{q+1}$ lies completely in $C$,
  every other component of $G[P_{\ge(q+1)}]$ (and its respective minor
  model) is not affected by this path. Therefore, it suffices to show how
  to find a pebbled minor model with the above properties for every
  component of $C-V(P_{q+1})$. Let $C'$ be such a component and let $M$ be
  the proper subgraph of $H-x$ associated with $C$. We show how to
  construct from $M$ a graph~$M'$ and a corresponding minor model with the
  appropriate properties for $C'$. Note that the vertex model~$H_a$ added
  in Case~2 automatically satisfies Conditions~\ref{it:1} and~\ref{it:peb}.

  We iteratively re-establish the properties for the vertex models $H_i$
  with $h_i\in V(M)$, in any order. Fix some $i$ with $h_i\in M$ and
  consider a path $E_{ij}$ such that the vertex $v_{ij}\in V(H_i)$ that is
  pebbled by $p_{ij}$ has no connection to $C'$. Let
  $E_{ij}=w_1,\ldots,w_s$, where $w_1=v_{ij}$. Let~$a$ be minimal such
  that~$w_a$ has a connection to $C'$, or let $a=s-1$ if no such vertex
  exists on~$E_{ij}$. We add all vertices $w_1,\ldots, w_a$ to~$H_i$. If
  $w_a$ has a connection to $C'$, we redefine $E_{ij}$ as the path
  $w_a,\ldots, w_s$ and place the pebble~$p_{ij}$ on $w_a$. If $w_a$ has no
  connection to $C'$, we delete the edge $\{h_i,h_j\}$ from~$M'$. If after
  fixing every path $E_{ij}$ for $H_i$ in the above way, $H_i$ has no
  connections to $C'$, we delete $h_i$ from~$M'$. Otherwise, if there are
  pebbles that do not lie on a vertex with a connection to $C'$, we place
  these pebbles on arbitrary vertices that are occupied by another pebble,
  that is, that have a connection to~$C'$.

  After performing these operations for every $H_i$, all conditions are
  satisfied. Condition~\ref{it:peb} is re-established for every $H_i$: if
  $h_i$ is not removed from $M'$, then every pebble that lies on a vertex
  that has no connection to $C'$ is pushed along a path until it lies on a
  vertex that does have a connection to $C'$, or finally, if there is no
  such connection on the path that it guards, it is placed at an arbitrary
  vertex that has a connection to $C'$. The operations on $H_i$ also
  re-establish Condition~\ref{it:3}(a) for one endpoint of $E_{ij}$. And
  after the operations are performed on $H_j$, Condition~\ref{it:3}(a) is
  re-established for $E_{ij}$. Furthermore, if $C'$ does not have a
  connection to a vertex model $H_i$, it may clearly be removed without
  violating Condition~\ref{it:all}. All other conditions are clearly
  satisfied.

  It remains to show that the graph $M$ for a component $C$ is always a
  proper subgraph of $H-x$. This however is easy to see. Assume that
  $M=H-x$ and all conditions are satisfied. By Condition~\ref{it:peb},
  every $H_i$, $1\le i\le k$, has a connection to $C$. Then, by adding $C$
  as a subgraph~$H_{k+1}$ to the minor model, we find $H$ as a minor, a
  contradiction.
\end{proof}

\begingroup
\def\thethm{\relax}
\begin{thm} \ {\rm(Theorem~\ref{thm:col})}\mbox{}\\*
  Let $H$ be a graph and $x$ a vertex of $H$. Set $h=|E(H-x)|$, and let
  $\alpha$ be the number of isolated vertices of $H-x$. Then for every
  graph $G$ that excludes $H$ as a minor, we have
  \begin{equation*}
    \col_r(G)\:\le\: h\cdot(2r+1)+\alpha.
  \end{equation*}
\end{thm}
\endgroup

\begin{proof}
  We strengthen the analysis in the proof of \Cref{lem:spd} by taking into
  account the special properties of the isometric paths decomposition
  constructed in the proof of \Cref{lem:minordecomp}.
  
  Let $\mathcal{P}=(P_1,\ldots, P_\ell)$ be an isometric paths
  decomposition of $G$ that was constructed as in the proof of
  \Cref{lem:minordecomp}, and let $L$ be an order defined from the
  decomposition. Let $v\in V(G)$ be an arbitrary vertex and choose~$q$ such
  that $v\in V(P_{q+1})$. Let $C$ be the component of $G[P_{\ge(q+1)}]$
  that contains $v$, and let $Q_1,\ldots,Q_m$, $1\le m\le q$, be the paths
  among $P_1,\ldots,P_q$ that have a connection to~$C$. By definition
  of~$L$, the vertices in $\Sreach_r[G,L,v]$ can only lie on
  $Q_1,\ldots,Q_m$ and on $P_{q+1}$.
  
  In the proof of \Cref{lem:minordecomp}, we associated with the component
  $C$ a pebbled minor model of a proper subgraph $M$ of $H-x$. The paths
  $Q_1,\ldots, Q_m$ were either associated with a vertex model~$H_i$
  representing a vertex $h_i$ of $M$, or with a path $E_{ij}$ representing
  an edge $e_{ij}$ of $M$. Just as in the proof of \Cref{lem:spd}, we can
  argue that
  $\bigl|\Sreach_r[G,L,v]\cap Q_j\bigr|\le \min\{\,|V(Q_j)|,\,2r+1\,\}$ for
  each path $Q_j$. However, the paths that lie inside a vertex model $H_i$
  can have only as many connections to $C$ as there are pebbles on it,
  since, by Condition~\ref{it:peb} of the proof of \Cref{lem:minordecomp},
  every connection of~$H_i$ to $C$ must be pebbled. Let $q$ be the number
  of paths $E_{ij}$ that have vertices connected to~$C$ and in the
  $r$-neighbourhood of $v$. By Condition~\ref{it:3}(c) from the proof, for
  every such path~$E_{ij}$ with endpoints~$v_i$ and $v_j$, the pebbles
  $p_{ij}$ and $p_{ji}$ lie on vertices $v_{ij}$ and~$v_{ji}$ such that the
  path $E'_{ij}=\{v_{ij},v_i\}+E_{ij}+\{v_j,v_{ji}\}$ is isometric. Thus
  $N_r[v]$ meets only at most~$h$ many paths $E'_{ij}$. It follows from
  \Cref{lem:spd} that $\col_r(G)\le h(2r+1)+\alpha$.
\end{proof}

\section{The generalised colouring numbers of planar graphs}
\label{sec:wcolplanar}

In this section we prove Theorems~\ref{thm:colg} and~\ref{thm:wcolg},
providing upper bounds for $\col_r(G)$ and $\wcol_r(G)$ when~$G$ is a graph
with bounded genus. Since for every genus~$g$ there exists a~$t$ such that
every graph with genus at most~$g$ does not contain $K_t$ as a minor, we
could use Theorems~\ref{thm:col} to obtain upper bounds for the generalised
colouring numbers of such graphs. But the bounds obtained in this section
are significantly better.

\subsection{The weak $r$-colouring number of planar graphs}\label{sub:5.1}

By a \emph{maximal planar graph} we mean a (simple) graph that is planar,
but where we cannot add any further edges without destroying planarity. It
is well known that a maximal planar graph~$G$ with $|V(G)|\ge3$ has a
unique plane embedding (up to the choice of the outer face), which is a
triangulation of the plane. We will use that implicitly regularly in what
follows.

We start by obtaining an upper bound for $\wcol_r(G)$ that is much smaller
than the bound given by \Cref{thm:col}. Our method for doing this again
uses isometric paths decompositions. For maximal planar graphs, we will
provide isometric paths decompositions of width at most~$2$. Using
\Cref{lem:spdwcol} and the fact that $\wcol_r(G)$ cannot decrease if edges
are added, we conclude that $\wcol_r(G)\le \binom{r+2}{2}\cdot(2r+1)\in
\Oof(r^3)$. In~\cite{Grohe15}, Grohe \emph{et al.}\ proved that for every
$r$ there is a graph~$G_{2,r}$ of tree-width 2 such that
$\wcol_r(G_{2,r})=\binom{r+2}{2}\in \Omof(r^2)$. Since graphs with
tree-width 2 are planar, this shows that the maximum of $\wcol_r(G)$ for
planar graphs is both in $\Omof(r^2)$ and~$\Oof(r^3)$.

\begin{lemma}\label{lem:wcolplanar}\mbox{}\\*
  Every maximal planar graph $G$ has an isometric paths decomposition of
  width at most 2.
\end{lemma}

\begin{proof}
  Fix a plane embedding of $G$. Since the proof is otherwise trivial, we
  assume $|V(G)|\ge4$.

  We will inductively construct an isometric paths decomposition
  $P_1,\dots,P_\ell$ such that each component $C$ of
  $G-\bigcup_{1\le j\le \ell}V(P_j)$ satisfies that the boundary of the
  region in which $C$ lies is a cycle in $G$ that has its vertices in
  exactly two paths from $P_1, \dots, P_\ell$.

  As the first path $P_1$, choose an arbitrary edge of the (triangular)
  outer face, and as $P_2$ choose the vertex of that triangle that is not
  contained in $P_1$. There is only one connected component in
  $G-(P_1\cup P_2)$, and it is in the interior of the cycle which has
  vertices $V(P_1)\cup V(P_2)$.

  Now assume that $P_1,\ldots, P_i$ have been constructed in the desired
  way, and choose an arbitrary connected component $C$ of
  $G-\bigcup_{1\le j\le i}V(P_j)$. Let $D$ be the cycle that forms the
  boundary of the region in which $C$ lies, and let $P_a,P_b$,
  $1\le a,b\le i$, be the paths that contain the vertices of $D$. Notice
  that at least one of these paths must have more than one vertex, and let
  $P_a$ be such a path.

  \begin{figure}
    \centering
    \includegraphics[origin=c,width=3 in,height=2 in]{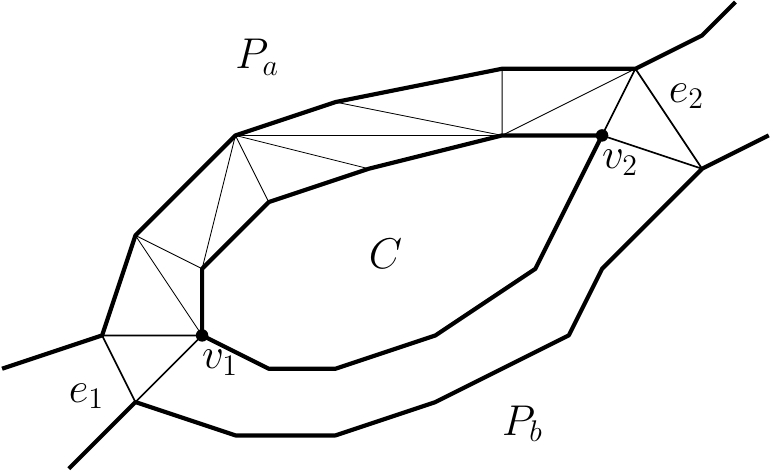}
    \caption{The path $P_{i+1}$ is chosen from the vertices of a connected
      component $C$.}
    \label{six}
  \end{figure}

  Since $P_a$ and $P_b$ are disjoint and isometric paths, $D$ must contain
  exactly two edges $e_1,e_2$ that do not belong to $P_a$ and $P_b$. (Of
  course, more than two edges can connect $P_a$ and $P_b$. But only two of
  them are on $D$.) Each of these edges belongs to a triangle in $G$ which
  is in the interior of $D$. By definition of $D$, the triangle that
  consists of $e_1$ and a vertex $v_1$ in the interior of~$D$ has the
  property that $v_1$ must lie in $C$. Similarly, the triangle that
  consists of~$e_2$ and a vertex $v_2$ in the interior of~$D$ has the
  property that~$v_2$ must lie in $C$ ($v_1=v_2$ is possible). See
  Figure~\ref{six} for a sketch of the situation.

  Any path $P$ in $C$ that connects~$v_1$ and $v_2$ has the property that
  every vertex of $C$ that is adjacent to $P_a$ is either in $P$ or in the
  region defined by $P_a$ and $P$ that does not contain $P_b$. Hence, as a
  next path $P_{i+1}$ we can take any isometric path in $C$ connecting
  $v_1$ and $v_2$.

  It is clear that any component $C'$ of $G-\bigcup_{1\le j\le i+1}V(P_j)$
  that was not already a component of $G-\bigcup_{1\le j\le i}V(P_j)$ is
  connected to at most two paths from $P_a, P_b, P_{i+1}$, and no such
  component is connected to both $P_a$ and $P_b$. To finish the
  construction of the decomposition we must prove that such a component
  $C'$ is connected to exactly two of these three paths. Let us assume that
  $C'$ lies in the interior of some cycle $D'$ contained in
  $V(P_a)\cup V(P_{i+1})$. Suppose for a contradiction that~$D'$ only has
  vertices from one of these paths, say from $P_a$. But since any cycle
  contains at least one edge not in $P_a$ and $D'$ has length at least~3,
  this implies that there is an edge between two non-consecutive vertices
  of $P_a$. This contradicts that $P_a$ was chosen as an isometric path.
  Exactly the same arguments apply when $C'$ lies in the interior of some
  cycle contained in $V(P_b)\cup V(P_{i+1})$.

  The isometric paths decomposition we constructed has width $2$, and thus
  the result follows.
\end{proof}

\clearpage
\begingroup
\def\thethm{\relax}
\begin{thm} \ {\rm(Theorem~\ref{thm:wcolg})}\mbox{}\\*
  For every graph $G$ with genus $g$, we have
  $\displaystyle\wcol_r(G)\le \biggl(2g+\binom{r+2}{2}\biggr)\cdot (2r+1)$.

  \smallskip\noindent
  In particular, for every planar graph $G$, we have
  $\displaystyle\wcol_r(G)\:\le\: \binom{r+2}{2}\cdot (2r+1)$.
\end{thm}
\endgroup

\begin{proof}
  We first prove the bound for planar graphs. According to
  \Cref{lem:wcolplanar}, maximal planar graphs have isometric paths
  decompositions of width at most $2$. Using \Cref{lem:spdwcol}, we see
  that any maximal planar graph $G$ satisfies
  $\wcol_r(G)\le \binom{r+2}{2}\cdot(2r+1)$. Since $\wcol_r(G)$ cannot
  decrease when edges are added, we conclude that any planar graph
  satisfies the same inequality.

  It is well known (see e.g.\ \cite[Lemma 4.2.4]{mohar01} or
  \cite{quilliot85}) that for a graph of genus $g>0$, there exists a
  non-separating cycle $C$ that consists of two isometric paths such that
  $G-C$ has genus $g-1$. We construct a linear order of $V(G)$ by starting
  with the vertices of such a cycle. We repeat this procedure inductively
  until all we are left to order are the vertices of a planar graph $G'$.
  We have seen that we can order the vertices of $G'$ in such a way that
  they can weakly $r$-reach at most $\binom{r+2}{2}\cdot(2r+1)$ vertices in
  $G'$. By \Cref{lem:deletevertices} and \Cref{lem:shortestpath} we see
  that any vertex in the graph can weakly $r$-reach at most $2g\cdot(2r+1)$
  vertices from the cycles we put first in the linear order. The result
  follows immediately.
\end{proof}

\subsection{The $r$-colouring number of planar graphs}

From \Cref{lem:wcolplanar} and \Cref{lem:spd}, we immediately conclude that
$\col_r(G)\le 3(2r+1)$ if $G$ is planar. This is already an improvement of
what we would obtain using Theorem~\ref{thm:col} with the fact that planar
graph do not contain $K_5$ or $K_{3,3}$ as a minor. Yet we can further
improve this by showing that $\col_r(G)\le 5r+1$, a bound which is tight
for $r=1$. The method we use to prove this again uses isometric paths, but
differs from the techniques we have used before because we will use
sequences of separating paths that are not disjoint.

Let~$G$ be a maximal planar graph and fix a plane embedding of~$G$. Let~$v$
be any vertex of~$G$ and let~$S$ be a lexicographic breadth-first search
tree of~$G$ with root~$v$. For each vertex~$w$, let $P_w$ be the unique
path in~$S$ from the root $v$ to $w$.

The following tree-decomposition $(T,(X_t)_{t\in V(T)})$ is a well-known
construction that has been used to show that the tree-width of a graph is
linear in its radius.
\begin{enumerate}[itemsep=0pt, topsep=1pt]
\item $V(T)$ is the set of faces of~$G$ (recall that all these faces are
  triangles);
\item $E(T)$ contains all pairs $\{t,t'\}$ where the faces~$t$ and~$t'$
  share an edge in~$G$ which is not an edge of~$S$;
\item for each face $t\in V(T)$ with vertices $\{a,b,c\}$, let
  $X_t=V(P_a)\cup V(P_b)\cup V(P_c)$.
\end{enumerate}

We define a linear order~$L$ on the vertices of~$G$ as follows. Let $t'$ be
the outer face of~$G$, with vertices $\{a,b,c\}$. We pick one of the paths
$P_a,P_b,P_c$, say $P_a$, arbitrarily as the first path and order its
vertices starting from the root $v$ and moving up to $a$. We pick a second
path arbitrarily, say~$P_b$, and order its vertices which have not yet been
ordered, starting from the one closest to $v$ and moving up to $b$. After
this, we do the same with the vertices of the third path~$P_c$.

We now pick the outer face as the root of the tree~$T$ from the
tree-decomposition and perform a depth-first search on~$T$. Each bag~$X_t$
contains the union of three paths, but at the moment $t$ is reached by the
depth-first search on~$T$, at most one of these paths contains vertices
which have not yet been ordered. We order the vertices of such a path
starting from the one closest to $v$ and moving up towards the vertex which
lies in $t$.

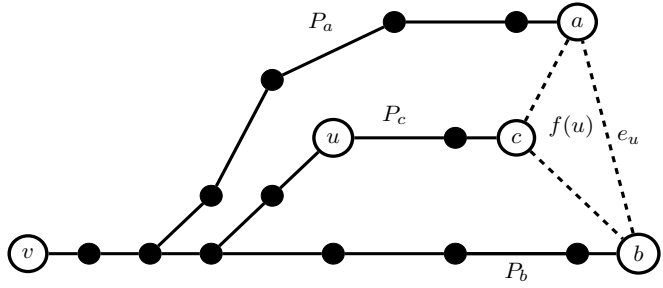
\begin{figure}
  \centering
  \resizebox{3.5 in}{1.5 in}{
    \begin{tikzpicture}
      \tikzstyle{dan} = [draw, ultra thick, circle, fill=none]
      \tikzstyle{ann} = [draw=none,fill=none]
      \tikzstyle{bar} = [draw, fill=black,circle]
      \node[dan] (n11) at (1,1) {$v$};
      \node[bar] (n21) at (2,1) {};
      \node[bar] (n31) at (3,1)  {};
      \node[bar] (n41) at (4,1)  {};
      \node[bar] (n61) at (6,1) {};
      \node[bar] (n81) at (8,1) {};
      \node[bar] (n101) at (10,1) {};
      \node[dan] (n111) at (11,1)  {$b$};
      \node[bar] (n42) at (4,2)  {};
      \node[bar] (n52) at (5,2)  {};
      \node[dan] (n63) at (6,3)  {$u$};
      \node[bar] (n83) at (8,3)  {};
      \node[dan] (n93) at (9,3)  {$c$};
      \node[ann] (n103) at (9.9,3.2)  {$f(u)$};
      \node[bar] (n54) at (5,4)  {};
      \node[bar] (n75) at (7,5)  {};
      \node[bar] (n95) at (9,5)  {};
      \node[dan] (n105) at (10,5)  {$a$};
      \node[ann] (n65) at (5.8,5)  {$P_a$};
      \node[ann] (n65) at (7,3.4)  {$P_c$};
      \foreach \from/\to in
      {n11/n21,n21/n31,n31/n41,n41/n61,n61/n81,n81/n101,n101/n111,
        n41/n52,n52/n63,n63/n83,n83/n93,
        n31/n42,n42/n54,n54/n75,n75/n95,n95/n105}
      \draw  [line width=1.5pt] (\from) -- (\to);
      \draw[line width=1.5pt] (n81) edge node[below] {$P_b$}(n101) ;
      \draw[line width=1.5pt] (n105) edge[dashed](n93) ;
      \draw[line width=1.5pt] (n93) edge[dashed](n111) ;
      \draw[line width=1.5pt] (n111) edge[dashed] node[right] {$e_u$}(n105) ;
    \end{tikzpicture}
  }
  \caption{Situation for a vertex~$u$ such that~$f(u)$ is not the outer
    face. Solid edges represent those in $G[X_{f(u)}]\cap S$. The
    cycle~$C(u)$ is the one contained in~$P_a \cup P_b \cup e_u$. The
    vertices $u$ and $c$ lie in $O(u)$.}
    \label{seven}
\end{figure}

For $u\in V(G)$, let~$f(u)$ be the first face (in the depth-first search
traversal of~$T$) for which the bag~$X_{f(u)}$ contains~$u$. If~$u$ is a
vertex for which~$f(u)$ is the outer face, then let~$C(u)$ be the cycle
formed by the three edges in the outer face. Otherwise, if~$f(u)$ is not
the outer face, then let~$e_u$ be the unique edge of~$f(u)$ not in~$S$ such
that the other face containing~$e_u$ was found by~$T$ before~$f(u)$, and
let~$C(u)$ be the cycle formed in~$S+e_u$. Finally, let~$O(u)$ be the set
of vertices lying in the interior of~$C(u)$. See Figure~\ref{seven} for a
sketch of the situation.

The following lemma tells us that if $f(u)$ is not the outer face, then the
paths in $X_{f(u)}$ separate~$u$ from any other smaller vertex in $L$.

\begin{lemma}\label{lem:carord}\mbox{}\\*
  For all $u \in V(G)$, we have that the vertices of $X_{f(u)}$ are
  smaller, with respect to~$L$, than all vertices in
  $O(u)\setminus X_{f(u)}$.
\end{lemma}

\begin{proof}
  If $f(u)$ is the outer face, then by the construction of $L$ the vertices
  of $X_{f(u)}$ are smaller than all other vertices in $V(G)$.

  Assume next that $f(u)$ is not the outer face and let $z<_L u$. If
  $f(z)=f(u)$, then also $X_{f(z)}=X_{f(u)}$ and by the definition of
  $X_{f(z)}$ we have that $z$ is contained in $X_{f(u)}$. If $f(z)\ne
  f(u)$, then it must be that $f(z)$ is a face which is encountered before
  $f(u)$ in the depth-first search of $T$. We know that $X_{f(z)}$ is the
  union of three paths. Assume for a contradiction that one of these paths,
  say~$P_1$, contains a vertex $x$ in $O(u)$. One of the endpoints of $P_1$
  is in~$f(z)$ and therefore cannot be in $O(u)$. The fact that $P_1$ has
  both vertices in $O(u)$ and vertices not in $O(u)$, means that there must
  be a vertex $w\ne v$ of $P_1$ in $C(u)$. Notice that $C(u)-e_u$ is a
  subset of two of the paths of $X_{f(u)}$. Therefore, $w$ also belongs to
  a path $P_2$ contained in $X_{f(u)}$ that does not have any of its
  vertices in $O(u)$. That means we have two paths between $w$ and the root
  $v$, one is a subpath of $P_1$ containing $x$, and the other one is a
  subpath of $P_2$ that does not contain $x$. However, any of the paths
  that form $X_{f(z)}$ and~$X_{f(u)}$ are paths of $S$, and this means we
  have found a cycle in $S$, a contradiction. We conclude that no path of
  $X_{f(z)}$ contains a vertex of $O(u)$ and so $z$ does not lie in $O(u)$.
\end{proof}

\noindent
We will use the ordering $L$ and~\Cref{lem:carord} to prove that
$\col_r(G)\le 5r+1$ for any planar graph~$G$. For the purpose of the
following proof, it is particularly important that $S$ is a
\emph{lexicographic} breadth-first search tree.

\begingroup
\def\thethm{\relax}
\begin{thm} \ {\rm(Theorem~\ref{thm:colg})}\mbox{}\\*
  For every graph $G$ with genus $g$, we have
  $\displaystyle\col_r(G)\le (4g+5)r+2g+1$.

  \smallskip\noindent
  In particular, for every planar graph $G$, we have
  $\displaystyle\col_r(G)\le 5r+1$.
\end{thm}
\endgroup

\begin{proof}
  Also this time we first prove the bound for planar graphs. Since
  $\col_r(G)$ cannot decrease when edges are added, we can assume that $G$
  is maximal planar. Therefore, we can order its vertices according to a
  linear order $L$ as defined above.

  Fix a vertex $u\in V(G)$ such that $f(u)$ is not the outer face, and let
  $a,b,c$ be the vertices of $f(u)$. Recall that this means that
  $X_{f(u)}=V(P_a)\cup V(P_b)\cup V(P_c)$. Choose $P_c$ to be the unique
  path of $X_{f(u)}$ containing $u$. Let $P_u$ be the subpath of $P_c$ from
  $u$ to the root $v$. Notice that $C(u)-e_u\subseteq P_a\cup P_b $. Then
  by~\Cref{lem:carord}, $P_a$ and $P_b$ separate $u$ from all smaller
  vertices not in $X_{f(u)}$. Therefore, using the definition of the
  ordering $L$, we see that all vertices in
  $V(P_a)\cup V(P_b) \cup V(P_u) \setminus\{u\}$ are smaller than $u$ in
  $L$, and that all the vertices in $O(u) \setminus V(P_u)$ are larger than
  $u$ in $L$. Hence, we have
  \begin{equation}\label{barrier}
    \Sreach_r[G,L,u]\: \subseteq\: N_r^G[u]\cap
    \bigl(V(P_a)\cup V(P_b)\cup V(P_u)\bigr).
  \end{equation}
  Since $S$ is a breadth-first search tree, using \Cref{lem:shortestpath},
  we see that $|N_r^G[u]\cap V(P_a)|\le 2r+1$ and
  $|N_r^G[u]\cap V(P_b)|\le 2r+1$. Also, by the definition of $L$ we have
  $|N_r^G[u] \cap V(P_u)|\le r+1$. These inequalities together
  with~(\ref{barrier}) tell us that $\bigl|\Sreach_r[G,L,u]\bigr|\le 5r+3$.

  In the remaining part of this proof we will show that in fact there are
  at least 2 fewer vertices in $\Sreach_r[G,L,u]$.

  We say that the \emph{level $d_u$} of a vertex $u$ is the distance~$u$
  has from $v$, \ie the height of $u$ in the breadth-first search tree $S$.
  For equality to occur in $|N_r^G[u] \cap V(P_a)|\le 2r+1$, there must be
  vertices $z_1,z_2\in V(P_a)$ in $N_r^G[u]$ such that the level of $z_1$
  in $S$ is $d_u-r$ and the level of $z_2$ is $d_u+r$. We will show that at
  most one of $z_1$ and $z_2$ can belong to
  $\Sreach_r[G,L,u]\setminus V(P_u)$.

  Suppose $z_2\in\Sreach_r[G,L,u]$ and let $P_2$ be a path from $u$ to
  $z_2$ that makes $z_2$ strongly $r$-reachable from $u$. Since $z_2$ is at
  level $d_u+r$, $P_2$ has length $r$ and all of its vertices must be at
  different levels of $S$. For any path $P$ with all of its vertices at
  different levels of $S$, we will denote by $P(d)$ the vertex of $P$ at
  level $d$. By definition of~$L$, we know that $P_a(d_u+i)<_L z_2$ for all
  $0\le i\le r-1$. This, together with the definition of $P_2$, tells us
  that $P_2$ cannot share any vertex with $P_a$ other than $z_2$. Moreover,
  the edge incident to $z_2$ in $P_2$ cannot belong to~$S$, because there
  already is an edge in $E(P_a)\subseteq E(S)$ joining a vertex at level
  $d_u+r-1$ to $z_2$. This means that the vertex $P_a(d_u+r-1)$ was found
  by the lexicographic breadth-first search~$S$ before the vertex
  $P_2(d_u+r-1)$. This in its turn implies that $P_a(d_u+r-2)$ was found by
  $S$ before $P_2(d_u+r-2)$. Continuing inductively we find that this is
  true for every level $d_u+i$, $0\le i\le r-1$. In a similar way, we can
  check that this implies that $S$ found the vertex $P_a(d_u-i)$ before the
  vertex $P_u(d_u-i)$, for $1\le i\le r$, whenever these vertices differ.
  In particular, $z_1$ was found before $P_u(d_u-r)$ if $z_1\notin V(P_u)$.

  Let us use this last fact to show that if $z_1\in\Sreach_r[G,L,u]$, then
  $z_1$ must also belong to $P_u$. We do this by assuming that
  $z_1\notin V(P_u)$. This tells us that the vertices $P_a(d_u-i)$ and
  $P_u(d_u-i)$ are distinct for all $0 \le i \le r$. Therefore, given that
  $z_1$ was found by $S$ before $P_u(d_u-r)$, there exists no edge between
  $z_1$ and $P_u(d_u-r+1)$, because if it did exist, then the edge joining
  $P_u(d_u-r)$ and $P_u(d_u-r+1)$ would not be in $S$. It follows that any
  vertex at level $d_u-r+1$ belonging to~$N(z_1)$ was found by $S$ before
  $P_u(d_u-r+1)$. By the same argument there is no edge between~$N(z_1)$
  and $P_u(d_u-r+2)$. Inductively, we find that for $0 \le i \le r-1 $, any
  vertex at level $d_u-r+i$ belonging to $N_i^G[z_1]$ was found by $S$
  before $P_u(d_u-r+i)$, and so there is no edge between~$N_i^G[z_1]$ and
  $P(d_u-r+i+1)$. But for $i=r-1$ this means that $u\notin N_r^G[z_1]$
  which implies that $z_1\notin\Sreach_r[G,L,u]$. Hence we can conclude
  that if $z_2\in\Sreach_r[G,L, u]$, then $z_1$ can only be strongly
  $r$-reachable from $u$ if it also belongs to $P_u$.

  Now suppose $z_1\in\Sreach_r[G,L,u]\setminus V(P_u)$, and let $P_1$ be a
  path from $u$ to $z_1$ that makes $z_1$ strongly $r$-reachable from $u$.
  Since $z_1$ is at level $d_u-r$, $P_1$ has length $r$ and all of its
  vertices are at different levels of $S$. Let $d_u-j$ be the minimum level
  of a vertex in $V(P_1)\cap V(P_u)$. Notice that $j<r$, since
  $z_1\notin V(P_u)$. Since $E(P_u)\subseteq E(S)$, it is clear that the
  vertex $P_u(d_u-j-1)$ was found by~$S$ before $P_1(d_u-j-1)$. This tells
  us that $P_u(d_u-j-2)$ was found before $P_1(d_u-j-2)$. Using induction,
  we can check that this will also be true for all levels $d_u-i$,
  $j+1\le i\le r$. In particular, this means that $P_u(d_u-r)$ was found
  before $z_1$. This implies that the lexicographic search found the vertex
  $P_u(d_u-i)$ before the vertex $P_a(d_u-i)$, for all $0\le i\le r$. Hence
  $S$ found~$u$ before~$P_a(d_u)$. Now suppose for a contradiction that
  there is a path $P_2$ that makes $z_2$ strongly $r$-reachable from $u$.
  The path $P_2$ can only intersect $V(P_a)$ at~$z_2$ and, since $u$ was
  found before~$P_a(d_u)$, it must be that $P_2(d_u+i)$ was found by $S$
  before $P_a(d_u+i)$, for $1\le i\le r-1$. Then the edge going from level
  $d_u+r-1$ to level $d_u+r$ in $P_a$ does not belong to $S$. This is a
  contradiction, given the definition of $P_a$.

  By the analysis above, we have that
  \begin{equation*}
    \bigl|\bigl(V(P_a)\setminus V(P_u)\bigr)\cap \Sreach_r[G,L,u]\bigr|\le
    2r.
  \end{equation*}
  In a similar way we can show that
  $\bigl|\bigl(V(P_b)\setminus V(P_u))\bigr)\cap \Sreach_r[G,L,u]\bigr|\le
  2r$. Then by~(\ref{barrier}) it follows that
  $\bigl|\Sreach_r[G,L,u]\bigr|\le 5r+1$ for this choice of $u$.

  We still have to do the case that $u$ is a vertex such that $f(u)$ is the
  outer face. We notice that it might be possible that when $u$ was added
  to the order~$L$, fewer than two paths reaching $f(u)$ had been ordered.
  In this case it is clear that
  $\bigl|\Sreach_r[G,L,u]\bigr|\le (2r+1)+(r+1)\le 5r+1$. If~$u$ is on the
  third chosen path leading from the root to the vertices of $f(u)$, then
  we can use the arguments above to show that
  $\bigl|\Sreach_r[G,L,u]\bigr|\le 5r+1$.

  \smallskip
  Having proved the bound on $\col_r(G)$ for planar graphs, the bound for
  graphs with genus $g>0$ can be easily proved following the same procedure
  as in the proof of Theorem~\ref{thm:wcolg} in the previous subsection.
\end{proof}

\bigskip\noindent
\textbf{Acknowledgement}

\smallskip\noindent
The authors would like to thank David Wood for pointing out that their
results imply an improvement on the upper bound of the acyclic chromatic
number of $K_t$-minor free graphs. They also thank the anonymous referees
for their corrections and suggestions.

\bibliographystyle{plain}
\bibliography{adm}

\end{document}